\documentclass[11pt,reqno]{amsart}

\usepackage[width=6.5in, height=9in, hcentering, vcentering]{geometry}
\usepackage{amssymb,amsmath,amsthm,mathrsfs,epsfig,times}
\usepackage[usenames]{xcolor}
\usepackage{verbatim}
\usepackage[pdftex]{hyperref}

\def\R{{\mathbb R}}
\def\*S{{\mathbb S}}

\newcommand{\cal}{\mathcal}

\DeclareMathOperator{\vol}{vol}
\DeclareMathOperator*{\esssup}{ess\;sup}

\newtheorem{theorem}{Theorem}
\newtheorem{proposition}{Proposition}[section]
\newtheorem{lemma}[proposition]{Lemma}

\begin{document}

\title{On a perturbation method for stochastic parabolic PDE.}

\author{Peter L. Polyakov}
\address{University of Wyoming \\ Department of Mathematics \\ 1000 E University Ave
\\ Laramie, WY 82071}
\curraddr{}
\email{polyakov@uwyo.edu}
\thanks{Author was partially supported by the NEUP program of
the Department of Energy}
\def\DATE{Version 3.\ Started 05.14.13}

\subjclass[2010]{Primary 65C30, 35K20}

\maketitle

\section{Introduction.}\label{intro}

In the present article we address two issues related to the perturbation method introduced
by Zhang and Lu in \cite{ZL}, and applied  to solving linear stochastic parabolic PDE. Those issues are:
the construction of the perturbation series, and its convergence. One of the problems considered
by Zhang and Lu is the following boundary value problem in a domain $D\subset \R^d$
\begin{equation}\label{BVP}
\begin{aligned}
&u_t-\nabla_x\cdot\left(e^{Y(x,\omega)}\nabla_x u(\omega,x,t)\right)=g(x,t),\\
&u(x,0)=u_0(x),\\
&u(x,t)\Big|_{bD}=0,
\end{aligned}
\end{equation}
where $Y(x,\omega)=Y_0(x)+Y^{\prime}(x,\omega)$ is a random process depending
on the spatial variable $x$ and on the random variable $\omega\in \Omega$
with $Y_0(x)=\langle Y\rangle$.\\
\indent
In their method Zhang and Lu consider two series representations: the first one is the
Taylor series representation of $e^x$ with respect to $x$, which combined with the Karhunen-Lo\`eve expansion of $Y^{\prime}(x,\omega)$ produces a series representation of $e^{Y(x,\omega)}$
with respect to random variables, and the second one is the series representation of the solution
$u(\omega,x,t)$ with respect to the same random variables. The method of expanding
an equation with respect to random variables from the Karhunen-Lo\`eve expansion has been
applied earlier within the framework of the {\it polynomial chaos} (see \cite{GS, X}). However,
in their work Zhang and Lu do not use the basis of mutually orthogonal polynomials in the space of all
polynomials of random variables, but rather the set of all monomials. The naturally arising question
in the case of a nonorthogonal basis is the uniqueness of the resulting series representation.
We address this question in Theorem~\ref{Independence}, where we prove the uniqueness
of the power series representation \eqref{Series} of a solution of \eqref{BVP}.

\indent
Another natural question arising in any asymptotic expansion is the question of convergence.
We address this question in Theorem~\ref{Estimate}, in which we prove the convergence
of the moments for the asymptotic series of a solution of problem \eqref{BVP} under
the assumption of smallness of the variance of $Y^{\prime}$.
Estimate \eqref{xi-Estimate} on the one hand shows that for the small enough variance
of $Y^{\prime}$ the resulting perturbation series \eqref{Series}
for a solution $u(\omega,x,t)$ converges in the probabilistic sense.
On the other hand the constant $C$ in this estimate increases with the increase of
the number of terms of Karhunen-Lo\`eve expansion used to approximate the random
process $Y^{\prime}(x,\omega)$. Therefore, an priori estimate \eqref{xi-Estimate}
of the rate of convergence of perturbation series \eqref{Series} becomes
worse as the number of terms in the Karhunen-Lo\`eve expansion increases.
We present the explaining argument in section~\ref{TheoreticalRate}.\\
\indent
The author would like to thank Victor Ginting and Don Estep for useful discussions on the
topics of the present article.

\section{Perturbation Series Expansion.}

\indent
Following Zhang and Lu \cite{ZL} we assume that $Y^{\prime}$ in \eqref{BVP} is a Gaussian random
process with the covariance function $C(x,y)$ defined on $\R^d\times \R^d$ by
\begin{equation}\label{Covariance}
C(x,y)=\sigma^2\exp\left(\frac{-|x-y|^2}{2\eta^2}\right).
\end{equation}
Then the following Karhunen-Lo\`eve expansion of $Y^{\prime}$ (\cite{CH, L})
provides an approximation
of $Y^{\prime}(x,\omega)$ in $L^2\left(\Omega\right)$ uniformly with respect to $x$:
\begin{equation}\label{K-L}
Y^{\prime}(x,\omega) = \sum_{i=1}^\infty \sqrt{\lambda_i} f_i(x) \xi_i(\omega),
\end{equation}
where $\left\{\xi_i(\omega)\right\}_1^{\infty}$ are mutually independent  Gaussian random
variables with zero mean, $\left\{\lambda_i\right\}_1^{\infty}$ is the set of eigenvalues
of the integral operator
\begin{equation*}
{\cal C}[f](x)=\int_D C(x,y)f(y)dy,
\end{equation*}
and $\left\{f_i\right\}_1^{\infty}$ is the set of normalized in $L^2(D)$ eigenfunctions.\\
\indent
Since the covariance function satisfies condition $C(x,x)=\sigma^2$,
from equality \eqref{K-L} using equality
$$\int_DC(x,x)dx=\sigma^2\cdot\vol(D),$$
we obtain the equality
\begin{equation*}
\sigma^2\cdot\vol(D)=\int_D C(x,x)dx
=\int_Ddx\int_{\Omega}Y^{\prime}(x,\omega)\cdot Y^{\prime}(x,\omega)d\omega
=\sum_{i=1}^{\infty}\lambda_i.
\end{equation*}
\indent
Then for the $k$-th partial sum of \eqref{K-L}
\begin{equation} \label{K-L-partial}
Y_k(x,\omega) = \sum_{i=1}^k \sqrt{\lambda_i} f_i(x) \xi_i(\omega).
\end{equation}
using the Cauchy-Schwarz inequality we obtain the following estimate:
\begin{equation}\label{Y-partial-Estimate}
\left|Y_k(x,\omega)\right|\leq \left(\sum_{i=1}^k\lambda_i\right)^{1/2}
\cdot\left(\sum_{i=1}^k\xi_i^2(\omega) f_i^2(x)\right)^{1/2}
\leq\sigma\cdot\left(\vol(D)\cdot\sum_{i=1}^k\xi_i^2(\omega) f_i^2(x)\right)^{1/2}.
\end{equation}
\indent
Apparently inspired by the estimate above, Zhang and Lu assume in \cite{ZL} that problem \eqref{BVP}
with $Y^{\prime}(x,\omega)$ replaced by $Y_k(x,\omega)$ admits a solution in the form
of a perturbation series
\begin{equation}\label{Series}
u(\omega,x,t)=u^{(0)}(x,t)+u^{(1)}(\omega,x,t)+\cdots,
\end{equation}
with the \lq\lq order of magnitude of $u^{(m)}$\rq\rq\ being proportional to $\sigma^m$.

\indent
To make the order of term $u^{(m)}$ more precise we follow \cite{ZL},
and assume that term $u^{(m)}$ is a homogeneous polynomial of order $m$ with respect
to variables $\xi_1,\dots,\xi_k$, i.e.
\begin{equation}\label{u_homogeneous}
u^{(m)}=\sum_{|I|=m}a_I(x,t)\xi^I,
\end{equation}
where $I=(i_1,\dots,i_k)$ is a multiindex, $|I|=i_1+\cdots+i_k$, and $\xi^I=\xi_1^{i_1}\cdots\xi_k^{i_k}$.\\
\indent
To obtain the uniqueness of a solution of the form \eqref{Series} we prove the linear
independence of products of mutually independent Gaussian random variables.

\begin{theorem}\label{Independence} Let
\begin{equation}\label{sum}
S(\xi)=\sum_{|I|=0}^n a_I\xi^J=\sum_{|I|=0}^n a_I\xi_1^{i_1}\cdots\xi_k^{i_k}
\end{equation}
be a linear combination of products of mutually independent standard normal random variables
$\xi_1,\dots,\xi_k$, where $I=(i_1,\dots,i_k)$ is a multiindex, and $|I|=i_1+\cdots+i_k$.\\
\indent
If $S(\xi)=0$, then $a_I=0$ for all multiindices in the sum.
\end{theorem}
\begin{proof}
Let ${\cal I}$ denote the set of all multiindices in the sum \eqref{sum}, for which
$a_I\neq 0$. Let $j_1=\max_{\cal I}\{i_1\}$. If $i_1=j_1$
for several multiindices in the sum \eqref{sum}, then we denote the set of those multiindices
by ${\cal I}_1$ and select among those the subset ${\cal I}_2$ consisting of the multiindices
of the terms with $i_2=j_2\stackrel{\rm def}{=}\max_{{\cal I}_1}\{i_2\}$.\\
\indent
Continuing this selection process we construct sets of multiindices
${\cal I}_1\supseteq{\cal I}_2\supseteq\cdots\supset{\cal I}_l$, with ${\cal I}_l$ consisting
of a unique multiindex
\begin{equation}\label{unique}
J=(j_1,\dots,j_l,j_{l+1},\dots,j_k)
\end{equation}
in the sum \eqref{sum}, such that
\begin{equation}\label{jConditions}
\left\{\begin{array}{ll}
j_r=\max_{{\cal I}_{r-1}}\{i_r\}\ \text{for}\ r=1,\dots,l,\vspace{0.1in}\\
j_l>i_l\ \text{for all}\ I\in{\cal I}_{l-1}.
\end{array}\right.
\end{equation}

\indent
In the Lemma below we prove an estimate for the term $a_J\xi^J$ that will be used in the
proof of Theorem~\ref{Independence}.

\begin{lemma}\label{IndexChoice} Let $J$ be the multiindex satisfying conditions
\eqref{jConditions}. Then there exists a multiindex
$M=\left(m_1,\dots,m_l,m_{l+1},\dots,m_k\right)$, such that inequality
\begin{equation}\label{Inequality}
\Big|a_J\cdot\left\langle\xi^M\cdot\xi^J\right\rangle\Big|\\
>2\sum_{I\neq J}^n\Big|a_I\cdot\left\langle\xi^M\cdot\xi^I\right\rangle\Big|
\end{equation}
is satisfied.
\end{lemma}
\begin{proof}
\indent
We prove the Lemma by induction with the following inductive statement. For any $1\leq r\leq l$
there exist large enough $m_1,\dots,m_r$, such that
\begin{equation}\label{InductiveInequality}
\sum_{I\in{\cal I}_r}\Big|a_I\cdot\left\langle\xi^M\cdot\xi^I\right\rangle\Big|
>2\sum_{I\notin{\cal I}_r}\Big|a_I\cdot\left\langle\xi^M\cdot\xi^I\right\rangle\Big|.
\end{equation}
To prove the starting point of the induction, i.e. inequality \eqref{InductiveInequality} for $r=1$,
we choose and fix indices $m_r$ for $r=l+1,\dots,k$
so that indices $j_r+m_r$ are even for the multiindex $J$ in \eqref{unique}. Then we choose 
indices $m_r$ for $r=1,\dots l$ so that
\begin{itemize}
\item
indices $j_r+m_r$ are even for the multiindex
$J$ in \eqref{unique},
\item
index $m_1$ is large enough for the inequality
\begin{equation*}
\sum_{I\in{\cal I}_1}\Big|a_I\cdot\left\langle\xi^M\cdot\xi^I\right\rangle\Big|
>2\sum_{I\notin{\cal I}_1}\Big|a_I\cdot\left\langle\xi^M\cdot\xi^I\right\rangle\Big|
\end{equation*}
to be satisfied.
\end{itemize}
It is possible because according to equality
\begin{equation*}
\left\langle\xi^n\right\rangle=n!!
\end{equation*}
from (\cite{C}, \S10.5) we have equality
\begin{equation}\label{TermExpectation}
\left\langle\xi^M\cdot a_J\xi^J\right\rangle
=a_J\prod_{r=1}^l\left(j_r+m_r\right)!!\prod_{r=l+1}^k\left(i_r+m_r\right)!!,
\end{equation}
and since $j_1>\max_{I\notin{\cal I}_1}\{i_1\}$, we obtain
$$\left(j_1+m_1\right)!!>m_1\cdot\left(i_1+m_1\right)!!$$
for any $I=(i_1,\dots,i_k)\notin{\cal I}_1$.\\
\indent
To prove the step of induction we assume that for a fixed $r<l$ the indices $m_1,\dots,m_r$ are chosen so that inequality \eqref{InductiveInequality} is satisfied .
Then applying inequality
$$\left(j_{r+1}+m_{r+1}\right)!!>m_{r+1}\cdot\left(i_{r+1}+m_{r+1}\right)!!$$
for any $I=(i_1,\dots,i_k)\notin{\cal I}_{r+1}$, we choose $m_{r+1}$ large enough for the equality
\eqref{InductiveInequality} to be satisfied for $r+1$.\\
\indent
Finally, for $r=l$ equalities \eqref{InductiveInequality} and \eqref{Inequality} coincide.
\end{proof}

\indent
Combining condition \eqref{Inequality} from Lemma~\ref{IndexChoice} with equality
\begin{equation*}
\Big\langle\xi^M\cdot S(\xi)\Big\rangle=0
\end{equation*}
for the multiindex $M$ constructed in Lemma~\ref{IndexChoice}, we obtain equality $a_J=0$.
Continuing the process described in Lemma~\ref{IndexChoice} we prove that $a_I=0$
for all multiindices in sum \eqref{sum}, since the number of terms in \eqref{sum} is finite.
\end{proof}

\section{Asymptotic Estimates.}

\indent
In this section we prove the convergence of the asymptotic series \eqref{Series} under the assumption
of smallness of $\sigma$. To simplify the computations in this section we restrict ourselves
to the case $L=\triangle$, i.e. $\langle Y\rangle=0$, though the computations are similar
for an arbitrary elliptic second order operator.\\
\indent
Substituting expressions \eqref{K-L-partial} and \eqref{Series} into equation \eqref{BVP},
and separating powers with respect to random variables $\xi_1,\dots,\xi_k$,
Zhang and Lu obtain the following sequence of boundary value problems for $m=0,1,\dots$.

\begin{lemma}\label{Sequence} (\cite{ZL})\ The sequence of functions $u^{(m)}$ from \eqref{sum}
satisfies the following sequence of boundary value problems for $m=0,1,\dots$
\begin{equation}\label{BVPSequence}
\begin{aligned}
&\left\{
\begin{array}{lll}
u^{(0)}_t(x,t)- \triangle u^{(0)}(x,t) = g(x,t),\vspace{0.1in}\\
u(x,0)=u_0(x),\vspace{0.1in}\\
u(x,t)\Big|_{bD}=0,
\end{array}\right.\\
&\hspace{0.6in}\cdots,\\
&\left\{
\begin{array}{lll}
u^{(m)}_t(x,t,\xi)-\triangle u^{(m)}(x,t,\xi) = g^{(m)}(x,t,\xi),\vspace{0.1in}\\
u^{(m)}(x,0,\xi)=0,\vspace{0.1in}\\
u^{(m)}(x,t,\xi)\Big|_{bD}=0,
\end{array}\right.
\end{aligned}
\end{equation}
where
\begin{multline}\label{right-hand}
g^{(m)}(x,t,\xi)=\nabla Y_k(x,\xi)\cdot \nabla u^{(m-1)}(x,t,\xi)\\
-\sum_{j=1}^m\frac{(-1)^j}{j!}Y_k^j(x,\xi) u_t^{(m-j)}(x,t,\xi)
+\frac{(-1)^m}{m!}g(x,t)Y_k^m(x,\xi).\\
\end{multline}
\end{lemma}

\indent
We continue with the following Lemma, in which we use estimate \eqref{Y-partial-Estimate}
to prove an estimate for $Y_k(x,\xi)=Y_k(x,\omega)$.

\begin{lemma}\label{C-1Estimate} The following estimate holds with some $C_1>0$
\begin{equation}\label{Y-Estimate}
\sup_{x\in D}\left(\left|Y_k(x,\xi)\right|+\sup_{r}\left|\frac{\partial Y_k}{\partial x_r}(x,\xi)\right|\right)
\leq C_1\cdot\sigma\left|\xi\right|,
\end{equation}
where $|\xi| \equiv \left(\xi_1^2 + \cdots + \xi_k^2\right)^{1/2}$.
\end{lemma}
\begin{proof} First we notice that for a bounded domain $D$ and a covariance
$C(x,y)\in C^{\infty}(D\times D)$ we have
\begin{equation}\label{EigenEstimate}
\sup_{x\in D}\left(\sup_{1\leq i\leq k}\left(|f_i(x)|
+\sup_{r}\left|\frac{\partial f_i}{\partial x_r}(x)\right|\right)\right)=C_0<\infty
\end{equation}
for some $C_0>0$.
Then from equality \eqref{K-L-partial} and estimate \eqref{EigenEstimate} we obtain
\begin{equation*}
\sup_{x\in D}\left(\left|Y_k(x,\xi)\right|+\sup_{r}\left|\frac{\partial Y_k}{\partial x_r}(x,\xi)\right|\right)
\leq C_0\sum_{i=1}^k \sqrt{\lambda_i}  |\xi_i|,
\end{equation*}
and using Cauchy-Schwarz inequality,
\begin{equation*}
\sup_{x\in D}\left(\left|Y_k(x,\xi)\right|+\sup_{r}\left|\frac{\partial Y_k}{\partial x_r}(x,\xi)\right|\right)
\leq C_0\left(\sum_{i=1}^k \lambda_i\right)^{1/2} \left(\sum_{i=1}^k\xi_i^2\right)^{1/2}\\
\leq \left(C_0\sqrt{\vol(D)}\right)\sigma |\xi|,
\end{equation*}
which implies \eqref{Y-Estimate} with $C_1=C_0\sqrt{\vol(D)}$.
\end{proof}

\indent
Theorem below represents the main result of the article - estimates on convergence of
series \eqref{Series}.

\begin{theorem}\label{Estimate}
The following estimates are satisfied by the sequence of solutions $u^{(m)}$ of the equations in
\eqref{BVPSequence}:
\begin{equation}\label{xi-Estimate}
\int_{\R^k}\left[\esssup_{0\leq t\leq T}\left\|u^{(m)}\right\|_{H^1_0(D)}(\xi)
+\left\|u^{(m)}_t\right\|_{L^2\left((0,T);L^2(D)\right)}(\xi)\right]\phi(\xi)P(\xi) d\xi\leq C^{2m+1}\sigma^m,
\end{equation}
where
$$\phi(\xi)=\frac{1}{(2\pi)^{k/2}}e^{-\frac{1}{2}|\xi|^2}$$
is the normal probability density function, $P(\xi)$ is a polynomial, $C$ is a constant depending
on $Y$, $P$ and $T$, and $\sigma^2$ is the variance of $Y$.
\end{theorem}

\indent
We prove Theorem~\ref{Estimate} as a corollary of the following proposition.

\begin{proposition}\label{StepEstimate}
For the sequence of solutions $u^{(m)}$ of equations in \eqref{BVPSequence} there exists a constant
$ C_2$ depending on $Y_k$, $g(x,t)$, and $T$, such that the following estimates hold
\begin{equation}\label{Fixed-xi-Estimate}
\esssup_{0\leq t\leq T}\left\|u^{(m)}\right\|_{H^1_0(D)}(\xi)
+\left\|u^{(m)}_t\right\|_{L^2\left((0,T);L^2(D)\right)}(\xi)\leq
\begin{cases}
C_2^{2m+1/2}\sigma^m|\xi|^m & \text{if }|\xi|\geq 1,\vspace{0.1in}\\
C_2^{2m+1/2}\sigma^m & \text{otherwise}.
\end{cases}
\end{equation}
\end{proposition}

\indent
{\it Proof of Proposition~\ref{StepEstimate}.} We prove Proposition~\ref{StepEstimate} by induction
with respect to $m$. Using estimate \eqref{Y-Estimate} and inductive assumption
in formula \eqref{right-hand} for the right-hand side of equation of order $m$ we obtain
\begin{multline}\label{m-Estimate}
\left\|g^{(m)}\right\|_{L^2((0, T); L^2(D))}(\xi)\\
= \left\|\nabla Y_k\cdot \nabla u^{(m-1)}-\sum_{j=1}^m\frac{(-1)^j}{j!}Y_k^j u_t^{(m-j)}
+\frac{(-1)^m}{m!}g(x,t)Y_k^m\right\|_{L^2\left((0,T);L^2(D)\right)}(\xi)\\
\leq \frac{\left\|g(x,t)\right\|_{L^2\left((0,T);L^2(D)\right)}}{m!}C_1^m\cdot \sigma^m|\xi|^m\\
+\left\|\sum_{r=1}^n\frac{\partial u^{(m-1)}}{\partial x_r}
\cdot \frac{\partial Y_k}{\partial x_r}\right\|_{L^2\left((0,T);L^2(D)\right)}(\xi)
+\sum_{j=1}^m \frac{1}{j!}\left\|u_t^{(m-j)}\cdot Y_k^j\right\|_{L^2\left((0,T);L^2(D)\right)}(\xi)\\
\leq \frac{\left\|g(x,t)\right\|_{L^2\left((0,T);L^2(D)\right)}}{m!}C_1^m\cdot \sigma^m|\xi|^m
+ \sqrt{dT}\cdot C_1\cdot C_2^{2(m-1)+1/2}\sigma^m|\xi|^m\\
+\sum_{j=1}^m \frac{1}{j!}\left\|u_t^{(m-j)}\cdot Y_k^j\right\|_{L^2\left((0,T);L^2(D)\right)}(\xi)\\
\leq \frac{\left\|g(x,t)\right\|_{L^2\left((0,T);L^2(D)\right)}}{m!}C_1^m\cdot \sigma^m|\xi|^m
+ \sqrt{dT}\cdot C_1\cdot C_2^{2(m-1)+1/2}\sigma^m|\xi|^m\\
+\sum_{j=1}^m \frac{1}{j!}C_1^j\cdot  C_2^{2(m-j)+1/2}\sigma^m|\xi|^m\\
\leq  C_2^{2m-1/2}\left(\|g(x,t)\|_{L^2\left((0,T);L^2(D)\right)}+\sqrt{dT}
+\sum_{j=1}^m \frac{1}{j!}\right)\sigma^m|\xi|^m
\leq  C_2^{2m}\sigma^m|\xi|^m,
\end{multline}
where we also assumed without loss of generality that
$$C_2^{1/2}>\max\left\{C_1,\left(\|g(x,t)\|_{L^2\left((0,T);L^2(D)\right)}+\sqrt{dT}
+\sum_{j=1}^m \frac{1}{j!}\right)\right\}.$$
\indent
To complete the step of induction we need to obtain necessary estimates for the solution
of equation of order $m$. To obtain those estimates for a fixed $\xi$ we use the following theorem
from \cite{E}.

\begin{theorem}\label{EvansEstimate} (Theorem 5 in \S 7.1 of \cite{E})\
Let $u_0\in H^1_0(D)$ and $g\in L^2\left((0,T);L^2(D)\right)$. Then a weak solution $u$ of
the parabolic equation with smooth coefficients
\begin{equation*}
\left\{\begin{array}{lll}
u_t-\nabla_x\cdot\left(A(x)\nabla_x u(x,t)\right)=g(x,t),\vspace{0.1in}\\
u(x,0)=u_0(x),\vspace{0.1in}\\
u(x,t)\Big|_{bD}=0,
\end{array}\right.
\end{equation*}
satisfies the following estimate
\begin{equation*}
\esssup_{0\leq t\leq T}\left\|u\right\|_{H^1_0(D)}
+\left\|u_t\right\|_{L^2\left((0,T);L^2(D)\right)}
\leq C\left(\left\|g\right\|_{L^2\left((0,T);L^2(D)\right)}+\left\|u_0\right\|_{H^1_0(D)}\right)
\end{equation*}
with the constant $C$ depending only on $D$, $T$, and $A(x)$.
\end{theorem}

\indent
Applying the above theorem  to equation of order $m$ in \eqref{BVPSequence},
using estimate \eqref{m-Estimate} and assuming without loss of generality that $ C_2^{1/2}>C$,
we obtain the first inequality in \eqref{Fixed-xi-Estimate}. The second inequality is proved similarly.
\qed\\

\indent
{\it Proof of Theorem~\ref{Estimate}.}
To prove Theorem~\ref{Estimate} we use estimate \eqref{Fixed-xi-Estimate} for a fixed $\xi$, and change the constant
$ C_2$ so that after integration with respect to $\xi$ we obtain estimate \eqref{xi-Estimate}. Namely, using estimate
\eqref{Fixed-xi-Estimate} we obtain
\begin{multline*}
\int_{\R^k}\left[\esssup_{0\leq t\leq T}\left\|u^{(m)}\right\|_{H^1_0(D)}
+\left\|u^{(m)}_t\right\|_{L^2\left((0,T);L^2(D)\right)}\right]P(\xi)e^{-\frac{1}{2}|\xi|^2} d\xi\\
\leq  C_2^{2m+1/2}\sigma^m\left[\int_{|\xi|<1}|P(\xi)|e^{-\frac{1}{2}|\xi|^2} d\xi
+\int_{|\xi|>1}|\xi|^m|P(\xi)|e^{-\frac{1}{2}|\xi|^2} d\xi\right].
\end{multline*}
Then, without loss of generality adjusting constant $ C_2$ so that
$$ C_2^{1/2}>\int_{|\xi|<1}|P(\xi)|e^{-\frac{1}{2}|\xi|^2} d\xi
+\int_{|\xi|>1}|\xi|^m|P(\xi)|e^{-\frac{1}{2}|\xi|^2} d\xi,$$
we obtain estimate \eqref{xi-Estimate}.
\qed

\section{Rate of Convergence.}\label{TheoreticalRate}

\indent
As it was mentioned in Introduction, estimate \eqref{xi-Estimate}, though providing
a substantiation for the separation of problem \eqref{BVP} into a sequence of problems \eqref{BVPSequence}, becomes worse as the number of terms in the Karhunen-Lo\`eve
expansion increases. We can see that on a simple 1-D example considered in \cite{ZL}.
Namely, for a 1-D stochastic process on the unit interval with covariance function as
in \eqref{Covariance} the corresponding eigenfunctions can be found explicitly \cite{ZL}:
$$f_n(x)=\frac{\left[\eta w_n\cos{(w_nx)}+\sin{(w_nx)}\right]}{\sqrt{(\eta^2 w_n^2+1)/2+\eta}},$$
where the sequence $w_n\to\infty$ is a sequence of positive roots of the equation
$$(\eta^2w_n^2-1)\sin{w}=2\eta w\cos{w}.$$
\indent
Computing the derivatives we obtain
$$f_n^{\prime}(x)=w_n\frac{\left[-\eta w_n\sin{(w_nx)}+\cos{(w_nx)}\right]}
{\sqrt{(\eta^2 w_n^2+1)/2+\eta}},$$
which implies that
$$\left\|f_n^{\prime}\right\|_{C([0,1]}\to\infty\ \text{as}\ n\to \infty,$$
and therefore in estimates \eqref{EigenEstimate} and \eqref{Y-Estimate}
we have $C_0,\ C_1\to \infty$ as $k\to \infty$.

\end{document}